\documentclass[twoside,centertags]{amsart}

\usepackage[cmtip,color,all]{xy}
\usepackage{amsmath}
\usepackage{amsfonts}
\usepackage{amssymb}
\usepackage{amsthm}
\usepackage{graphicx}
\usepackage{mathrsfs}

\newtheorem{theorem}{Theorem}[section]
\newtheorem*{maintheorem}{Theorem}

\newtheorem*{Pseudopullback Theorem}{Pseudopullback Theorem}
\newtheorem*{Cofibration Theorem}{Cofibration Theorem}
\newtheorem{corollary}[theorem]{Corollary}
\newtheorem{lemma}[theorem]{Lemma}
\newtheorem{proposition}[theorem]{Proposition}
\theoremstyle{definition}

\theoremstyle{remark}
\newtheorem{remark}[theorem]{Remark}
\newtheorem{example}[theorem]{Example}

\newcommand{\cyl}{\operatorname{Cyl}}

\newcommand\cc{\mathcal {C}}

\newcommand\cf{\mathcal {F}}

\newcommand\ck{\mathcal {K}}
\newcommand\cl{\mathcal {L}}

\newcommand\cm{\mathcal {M}}

\newcommand\cp{\mathcal {P}}
\newcommand\cw{\mathcal {W}}
\newcommand\cx{\mathcal {X}}

\newcommand\modR{\mathrm{Mod}(R)}

\begin{document}

\title{The accessibility rank of weak equivalences} %

\author{G. Raptis}
\address{Fakult\"{a}t f\"{u}r Mathematik,
Universit\"{a}t Regensburg, 
93040 Regensburg, Germany}
\email{georgios.raptis@mathematik.uni-regensburg.de}

\author{J. Rosick\'y}%
\address{Department of Mathematics and Statistics, 
Masaryk University, Faculty of Sciences, 
Kotl\'{a}\v{r}sk\'{a} 2, 60000 Brno, Czech Republic}
\email{rosicky@math.muni.cz}

\begin{abstract}
We study the accessibility properties of trivial cofibrations and weak equivalences in a combinatorial 
model category and prove an estimate for the accessibility rank of weak equivalences. In particular, 
we show that the class of weak equivalences between simplicial sets is finitely accessible. 
\end{abstract}

\maketitle 

\section{Introduction}

A well-known and useful property of combinatorial model categories is that their classes of weak 
equivalences are accessible. Although each of the various proofs of this result can also give some 
estimate for the accessibility rank, these estimates are generally not the best possible. The purpose of 
this paper is to address the issue of determining good estimates for the accessibility 
rank of weak equivalences in cases of interest. Our main result is the following theorem. 

\begin{maintheorem}(see Theorem \ref{thmA})
Let $\mathcal{M}$ be a $\kappa$-combinatorial model category. Suppose that the following are satisfied:
\begin{itemize}
\item[(i)] there is a set $\mathcal{A}$ of $\kappa$-presentable cofibrant objects such that every object in $\mathcal{M}$ 
is a $\kappa$-filtered colimit of object in $\mathcal{A}$. 
\item[(ii)] every $\kappa$-presentable cofibrant object in $\mathcal{A}$ has a (not necessarily functorial) cylinder object which is again $\kappa$-presentable.
\item[(iii)] if a composite morphism $i' = q i$ is a cofibration between cofibrant objects, then $i$ is a formal cofibration. 
\end{itemize}
Then the full subcategory $\cw$ of $\mathcal{M}^{\to}$ spanned by the class of weak equivalences is $\kappa$-accessible.
\end{maintheorem}

The $\kappa$-accessibility of $\cw$ in $\cm^{\to}$ means the following: first, $\cw$ is closed under $\kappa$-filtered colimits and, secondly, given 
a weak equivalence $f: X \to Y$ in $\cm$ and a morphism $g:K\to L$ between $\kappa$-presentable objects, then every commutative square
\[
\xymatrix{
K \ar[r] \ar[d]^g & X \ar[d]^f \\
L \ar[r] & Y
}
\]
admits a factorization
\[
\xymatrix{
K \ar[r] \ar[d]^g & K^\prime \ar[d]^h \ar[r] & X \ar[d]^f \\
L \ar[r] & L^\prime \ar[r] & Y
} 
\]
where $h: K' \to L'$ is weak equivalence between $\kappa$-presentable objects. The first claim is the easy part of the theorem and its proof does 
not require the additional assumptions (i)-(iii). The proof of the second claim is more difficult and requires the auxiliary assumptions (i)-(iii). 

So far, proofs of the accessibility of the weak equivalences in a general combinatorial model category have made essential use of 
the Pseudopullback Theorem. The Pseudopullback Theorem, which will be recalled in more detail in Section 2, says that pseudopullbacks of 
accessible categories and functors are again accessible categories. In addition, the proof of this theorem can be used to infer an estimate for 
the accessibility rank of the pseudopullback based on the accessibility data of the diagram. By expressing the category of weak equivalences 
$\cw$ as such a pseudopullback, it is then possible to obtain estimates for its accessibility rank. Using this method, it is not difficult to show, for example, that the category of weak 
equivalences of simplicial sets, $\cw \subset \mathcal{SS}et^{\to}$, is $\aleph_1$-accessible. More generally, applying this method to get an
estimate for the accessibility rank of the weak equivalences in a $\kappa$-combinatorial model category will generally produce a regular 
cardinal strictly greater than $\kappa$. 

Our proof uses different methods which we think are also of independent interest. An important ingredient is the fat small-object argument 
which was introduced in \cite{MRV}. Based on this, we first prove in Section 3 that \textit{every trivial 
cofibration in a $\kappa$-combinatorial model category is a $\kappa$-directed colimit of trivial cofibrations between $\kappa$-presentable 
objects} (Corollary \ref{triv-cof}). This is also the first step towards the proof of the main theorem (Theorem \ref{thmA}) in Section 4. 

In Section 5 we discuss examples to which our theorem applies. These include the following:
\begin{itemize} 
\item the category of weak equivalences of simplicial sets is finitely accessible (Corollary \ref{corA}). This generalizes easily
to combinatorial model categories whose underlying category is a Grothendieck topos and the cofibrations are the monomorphisms (Corollary \ref{corB}). 
We note that a different proof of the finite accessibility of the weak equivalences of simplicial sets appeared subsequently 
in \cite{BS}.
\item the category of quasi-isomorphisms of chain complexes of $R$-modules is finitely accessible if $R$ is semi-simple (Corollary \ref{corC}). 
\item the category of stable isomorphisms of $R$-modules is finitely accessible if $R$ is a Noetherian Frobenius ring (Corollary \ref{corD}). 
\end{itemize}
Naturally, the accessibility rank is especially interesting in the cases where it happens to be $\aleph_0$, however, our main result will generally 
predict estimates in other cases too. We end Section 5 with some comments on the cases of diagram model categories and left Bousfield localizations. 

\

\noindent \textbf{Acknowledgements.} We are grateful to the referee for pointing out a simplification of our original proof which actually also led to a 
more general statement requiring less than our original assumptions. The first-named author would like to warmly acknowledge the hospitality of the Department 
of Mathematics in Brno, which he enjoyed during several visits while this work was being done. The second-named author was supported by GA\v{C}R 201/11/0528.

\section{Background material and preliminaries} \label{background}

Combinatorial model categories were defined by J. H. Smith to be
model categories $\cm$ which are locally presentable and cofibrantly generated. The latter means that both cofibrations 
and trivial cofibrations are cofibrantly generated by a set of morphisms. There is always a regular
cardinal $\kappa$ such that $\cm$ is locally $\kappa$-presentable and the generating cofibrations and trivial cofibrations are morphisms 
between $\kappa$-presentable objects. In this case we say that $\cm$ is $\kappa$-combinatorial. For instance, the standard model category of
simplicial sets, denoted $\mathcal{SS}et$, is finitely combinatorial (= $\aleph_0$-combinatorial). 

In a $\kappa$-combinatorial model category $\cm$, the `replacement-by-fibration' functor $R_{\mathrm{fib}}:\cm^{\to}\to\cm^{\to}$ that arises from the (trivial cofibration, fibration)-factorization, 
as given by the small-object argument, preserves $\kappa$-filtered colimits but it does not preserve $\kappa$-presentable objects in general. By the Uniformization 
Theorem (see \cite[2.4.9]{MP}, or \cite[2.19]{AR}), there is a regular cardinal $\lambda\geq\kappa$ such that $R_{\mathrm{fib}}$ preserves $\lambda$-presentable objects. This cardinal is important because
it makes the category of fibrations $\lambda$-accessible - any fibration is a $\lambda$-filtered colimit in $\cm^{\to}$ of fibrations between 
$\lambda$-presentable objects. To see this, consider a fibration $p: X \to Y$ in $\cm$ and express it as a $\lambda$-filtered 
colimit of morphisms $f_i : X_i \to Y_i$ between $\lambda$-presentable objects. Then $R_{\mathrm{fib}}(p)$ is a $\lambda$-filtered colimit of 
fibrations between $\lambda$-presentable objects. Since $p$ is a retract of $R_{\mathrm{fib}}(p)$, it follows that $p$ 
can also be expressed in this way (see the proof of \cite[2.3.11]{MP}). 

\begin{example} 
Let $\cm = \mathcal{SS}et$ and $R_{\mathrm{fib}}: \mathcal{SS}et^{\to} \to \mathcal{SS}et^{\to}$ the standard replacement by a fibration as given 
by the small-object argument. The functor $R_{\mathrm{fib}}$ preserves filtered colimits and sends finitely presentable objects to $\aleph_1$-presentable
ones. It follows that $R_{\mathrm{fib}}$ preserves $\aleph_1$-presentable objects since every $\aleph_1$-presentable object is an 
$\aleph_1$-small filtered colimit of finitely presentable objects (see \cite[2.3.11]{MP}, \cite[2.15]{AR}).
\end{example}

Analogously, one can consider instead the `replacement-by-trivial fibration' functor $R_{\mathrm{wfib}}: \cm^{\to} \to \cm^{\to}$ that comes from the standard (cofibration, trivial fibration)-factorization in $\cm$ and ask for the smallest regular cardinal $\lambda$ such
that $R_{\mathrm{wfib}}$ preserves $\lambda$-filtered colimits and $\lambda$-presentable objects. In $\mathcal{SS}et$, the smallest such $\lambda$ is again $\aleph_1$, which then shows that the category $\cf \cap \cw$ of 
trivial fibrations in $\mathcal{SS}et^{\to}$ is $\aleph_1$-accessible. 

Combining these observations, the accessibility of the class of weak equivalences, and an estimate for the accessibility rank, can be deduced from the following theorem. 

\begin{Pseudopullback Theorem}\label{pspbth}
Let $\lambda$ be a regular cardinal and $\ck$, $\cl$ and $\cm$ $\lambda$-accessible categories which admit $\kappa$-filtered 
colimits for some $\kappa < \lambda$. Let $F: \ck \to \cl$ and $G: \cm \to \cl$ be functors which preserve $\kappa$-filtered colimits and $\lambda$-presentable objects. Consider the pseudopullback
\[
\xymatrix@=4pc{
\cp \ar[d] \ar[r] & \cm \ar[d]_{G} \\
\ck \ar[r]^{F} & \cl \\
}
\]
Then $\cp$ is $\lambda$-accessible and has $\kappa$-filtered colimits.
\end{Pseudopullback Theorem}
 
This theorem is essentially shown in \cite[3.1]{CR}. Although the statement of \cite[3.1]{CR} includes stronger accessibility 
properties, exactly the same proof applies here too. H.-E. Porst has recently informed us that this theorem follows from the unpublished
paper \cite{U}.
 
The category of weak equivalences can be expressed as a pullback 
\[
\xymatrix@=4pc{
\cw \ar[d] \ar[r] & \cf \cap \cw \ar[d]_{G} \\
\cm^\to \ar[r]^{R_{\mathrm{fib}}} & \cm^\to \\
}
\]
where $G$ is the full embedding of the trivial fibrations. Since $G$ is transportable, this 
pullback is equivalent to a pseudopullback (see \cite[5.1.1]{MP}). Then the above Pseudopullback Theorem applied 
to this pseudopullback gives an estimate for the accessibility rank of $\cw$. Let us add that the accessibility 
of $\cw$ was claimed by J. H. Smith and proofs (also using the Pseudopullback Theorem) can be found in 
\cite[Corollary A.2.6.6]{L} and \cite{R}. 

\begin{proposition} \label{classical-estimate}
Let $\mathcal{M}$ be a $\kappa$-combinatorial model category and $\lambda > \kappa$ a regular cardinal such that  
both $R_{\mathrm{fib}}$ and $R_{\mathrm{wfib}}$ preserve $\lambda$-presentable objects. Then the category 
$\cw$ of weak equivalences, as a full subcategory of $\cm^{\to}$, is $\lambda$-accessible and admits $\kappa$-filtered
colimits.
\end{proposition}

\begin{example}
The discussion above and the last proposition show that the full subcategory $\cw$ of weak equivalences in $\mathcal{SS}et^{\to}$ is $\aleph_1$-accessible. 
\end{example}

This means that the estimate that we can get with this method for the accessibility rank of $\cw$  in a $\kappa$-combinatorial model category $\cm$ is going to be 
strictly greater than $\kappa$. A way of improving this estimate was opened in \cite{MRV} where the idea of good colimits from \cite{L} was used to prove 
that every cofibrant object in a $\kappa$-combinatorial model category is a $\kappa$-directed colimit of $\kappa$-presentable cofibrant objects. In Section 3, 
we will show how this can be used to prove that every trivial cofibration in a $\kappa$-combinatorial model category is a $\kappa$-directed colimit of trivial 
cofibrations between $\kappa$-presentable objects. This is the first step of the proof of our main theorem in Section 4, but the rest of the proof requires more 
assumptions on $\cm$.

\begin{remark}
It was shown in \cite{Ra2} that for every combinatorial model category $\cm$ there is an accessible functor $F: \cm \to \mathcal{SS}et$ such that a morphism 
$f$ in $\cm$ is a weak equivalence if and only if $F(f)$ is a weak equivalence in $\mathcal{SS}et$. Then the accessibility of $\cw$ follows from the Pseudopullback 
Theorem and the accessibility of the weak equivalences in $\mathcal{SS}et$. We like to point out and correct a small error in the construction of $F$ in \cite{Ra2}. 
We recall that $F$ was defined by a composition of functors as follows, 
$$\cm \stackrel{R}{\to} \cm \stackrel{G}{\to} \mathcal{SS}et^C \stackrel{u^*}{\to} \mathcal{SS}et^{\mathrm{Ob}C} \to \mathcal{SS}et,$$
where $R$ is an accessible fibrant replacement functor, $G$ is the right Quillen functor associated with a small presentation of $\cm$ \cite{D}, $u^*$ is the 
restriction functor and the last functor is an accessible functor that detects the (pointwise) weak equivalences of (pointwise) Kan fibrant objects. The last functor 
was wrongly chosen in \cite{Ra2} to be the product functor: this functor does not detect the weak equivalences when the empty (simplicial) set appears in the product! 
However, there are clearly other functors that have the required property, e.g., the coproduct. 
\end{remark}

\section{Trivial cofibrations} \label{fat-small-object}

Combinatorial categories, introduced in \cite{MR}, are locally presentable categories $\ck$
equipped with a class of morphisms $\mathrm{cof}(\ck)$, called cofibrations, which is cofibrantly 
generated by a set of morphisms. This basic categorical structure is inspired by combinatorial model 
categories, which can be viewed as combinatorial categories in (at least) two different ways, and the 
aim is to analyze through this generalization the abstract properties of cofibrant generation by focusing 
on a single cofibrantly generated class  of morphisms. 

In analogy with combinatorial model categories, we say that a combinatorial category $(\ck, \mathrm{cof}(\ck))$
is $\kappa$-combinatorial if it is locally $\kappa$-presentable and there is a set of generating cofibrations $I$ between $\kappa$-presentable objects. 
In what follows, $\ck_\kappa$ will denote the full subcategory of $\ck$ consisting of $\kappa$-presentable objects. 
Every locally presentable category carries the trivial combinatorial structure where every morphism is a cofibration. The following result is \cite[2.4]{MR}
but we will recall the proof.

\begin{lemma} \label{trivial} 
Let $\ck$ be a locally $\kappa$-presentable category. Then the trivial combinatorial structure on $\ck$ is $\kappa$-combinatorial.
\end{lemma}
\begin{proof}
If a morphism $g$ has the right lifting property with respect to all morphisms between $\kappa$-presentable
objects then $g$ is both a $\kappa$-pure monomorphism and a $\kappa$-pure epimorphism. Thus $g$ is both a regular monomorphism and an epimorphism
(see \cite{AHT}, \cite{AR1}), which means that $g$ is an isomorphism. Hence any morphism has the left lifting property with respect to $g$.
\end{proof}

We say that an object $K$ of a combinatorial category is cofibrant if the unique morphism $0\to K$ from an initial object 
is a cofibration. One of the main results of \cite{MRV} shows that \emph{every cofibrant object in a $\kappa$-combinatorial category is a $\kappa$-directed colimit of $\kappa$-presentable cofibrant objects} \cite[5.1]{MRV}. A consequence of this is the next theorem.

\begin{theorem} \label{cofibration-thm}
Let $\ck$ be a $\kappa$-combinatorial category. Then every cofibration is a $\kappa$-directed colimit of cofibrations between $\kappa$-presentable objects.
\end{theorem}
\begin{proof}
Let $\mathcal{I}\subseteq\ck_\kappa^{\to}$ be a generating set of cofibrations. Consider the following sets of morphisms in $\ck^{\to}$:
$$\chi_1 = \{(\mathrm{id}: A \to A) \stackrel{(g,g)}{\longrightarrow} (\mathrm{id}: B \to B) \colon g \in \mathcal{K}_{\kappa}^{\to} \}$$
and 
$$\chi_2 = \{(\mathrm{id}: A \to A) \stackrel{(\mathrm{id},i)}{\longrightarrow} (i: A \to B) \colon i \in \mathcal{I} \}.$$
Then the set $\chi_1 \cup \chi_2$ generates a $\kappa$-combinatorial structure on $\ck^{\to}$. By \cite[5.1]{MRV}, it suffices to show that its cofibrant objects are precisely the cofibrations in $\ck$. It is easy to see that every cofibrant object in $\ck^{\to}$ is a cofibration in $\ck$. Conversely, let $f: X \to Y$ be a cofibration in $\mathcal{K}$. An immediate consequence of Lemma \ref{trivial} is that the object $(\mathrm{id}: X \to X)$ is cofibrant with respect to $\chi_1$. Moreover, the morphism in $\mathcal{K}^{\to}$: 
$$(\mathrm{id}, f) \colon (\mathrm{id}: X \to X) \longrightarrow (f: X \to Y)$$
is a cofibration with respect to $\chi_2$. Therefore $f$ is a cofibrant object with respect to $\chi_1 \cup \chi_2$ and 
the result follows.  
\end{proof}

\begin{corollary} \label{triv-cof}
Let $\mathcal{M}$ be a $\kappa$-combinatorial model category. Then every trivial cofibration in $\mathcal{M}$ 
is a $\kappa$-directed colimit of trivial cofibrations between $\kappa$-presentable objects.
\end{corollary}
\begin{proof}
It suffices to apply Theorem \ref{cofibration-thm} to the $\kappa$-combinatorial category defined by the 
locally presentable category $\cm$ together with the class of trivial cofibrations of the model structure on $\cm$. 
\end{proof}

\begin{remark} \label{alternative-arg}
We point out that the proof of Corollary \ref{triv-cof} does not require the existence of a 
generating set of cofibrations between $\kappa$-presentable objects. A related argument that uses this assumption 
but has some other advantages is as follows. Let 
$\mathcal{I}$ denote a generating set of cofibrations between $\kappa$-presentable objects and replace $\chi_1$ with 
$$\chi'_1 = \{(\mathrm{id}: A \to A) \stackrel{(i,i)}{\longrightarrow} (\mathrm{id}: B \to B) \colon i \in \mathcal{I} \}$$
Then $(\mathrm{id}: X \to X)$ is $\chi'_1$-cofibrant if $X$ is cofibrant in $\mathcal{M}$. The same argument as above then shows that a 
trivial cofibration $f: X \to Y$ between \emph{cofibrant} objects is a $\kappa$-directed colimit of a diagram $F: P \to \mathcal{M}^{\to}$ whose values are
trivial cofibrations between $\kappa$-presentable \emph{cofibrant} objects. Another advantage of this argument is that if $X$ and $f$ are cellular
or $\kappa$ is uncountable, then $F$ can be chosen so that for all $p \in P$, the morphism $F(p) \to f$ in $\mathcal{M}^{\to}$ is given by cofibrations (see \cite{MRV}, esp. the proofs of 5.1 and 5.2). We recall that a morphism 
in a combinatorial category, cofibrantly generated by a set $\cx$, is called \textit{cellular} if it can be obtained from $\cx$
by transfinite compositions of pushouts. Cofibrations are then retracts of cellular morphisms.
\end{remark}

Corollary \ref{triv-cof} has the following consequence which may be useful in concrete situations. 

\begin{corollary} \label{triv-cof-fact}
Let $\mathcal{M}$ be a $\kappa$-combinatorial model category and 
\[
\xymatrix{
K \ar[r]^u \ar[d]^i & X \ar[d]^f \\
L \ar[r]^v & Y
}
\]
a commutative square where $i$ is a cofibration between $\kappa$-presentable objects and $f$ is a weak equivalence. Then there is a factorization
\[
\xymatrix{
K \ar[r] \ar[d]^i & K^\prime \ar[d]^j \ar[r] & X \ar[d]^f \\
L \ar[r] & L^\prime \ar[r] & Y
} 
\]
where $j$ is a trivial cofibration between $\kappa$-presentable objects.
\end{corollary}
\begin{proof}
Let 
$$
f : X \xrightarrow{\quad  f_1\quad} Z \xrightarrow{\quad f_2\quad} Y
$$
be a (trivial cofibration, fibration)-factorization of $f$. Then $f_2$ is also a weak equivalence. Since $i$ is a cofibration, there is a 
morphism $l: L\to Z$ making the diagram commutative 
\[
 \xymatrix{
 K \ar[r]^{f_1 u} \ar[d]^i & Z \ar[d]^{f_2} \\
 L \ar[r]^v \ar[ur]^l & Y
 }
\]
Thus, we get a morphism $(u,l): i \to f_1$ in $\cm^{\to}$. By Corollary \ref{triv-cof}, there is a factorization
\[
\xymatrix{
K \ar[r]^{u_1} \ar[d]^i & K^\prime \ar[d]^j \ar[r]^{u_2} & X \ar[d]^{f_1} \\
L \ar[r]^{l_1} & L^\prime \ar[r]^{l_2} & Z
} 
\]
where $u = u_2u_1$, $l =l_2 l_1$ and $j$ is a trivial cofibration between $\kappa$-presentable objects. Thus
\[
\xymatrix{
K \ar[r]^{u_1} \ar[d]^i & K^\prime \ar[d]^j \ar[r]^{u_2} & X \ar[d]^{f} \\
L \ar[r]^{l_1} & L^\prime \ar[r]^{f_2 l_2} & Y
} 
\]
is the required factorization.
\end{proof}

\section{Weak Equivalences} \label{main-proof}

Let $\mathcal{M}$ be a $\kappa$-combinatorial model category. Trivial cofibrations in $\cm$ are not closed 
under $\kappa$-filtered colimits in general. On the other hand, weak equivalences are accessibly embedded, as follows from the arguments 
of Section 2 and Proposition \ref{classical-estimate}. These arguments required that both cofibrations and trivial cofibrations are cofibrantly 
generated by sets of morphisms between $\kappa$-presentable objects. A more direct proof can be found in \cite[Proposition 7.3]{D}. We include a  slightly improved version of that proof which only requires this for the cofibrations. 
Furthermore, the proof that follows does not require that $\mathcal{M}$ is locally presentable. 

\begin{proposition} \label{acc-embedded}
Let $\mathcal{M}$ be a combinatorial model category and $\mathcal{I}$ a generating set of cofibrations between 
$\kappa$-presentable objects. Then the class of weak equivalences in $\mathcal{M}$ is closed 
under $\kappa$-filtered colimits in $\mathcal{M}^{\to}$. 
\end{proposition}
\begin{proof} 
Let $C$ be a small $\kappa$-filtered category and consider the category of $C$-diagrams 
$\mathcal{M}^C$ with the projective model structure where weak equivalences and fibrations 
are defined pointwise. Then the colimit functor 
$$\mathrm{colim}_C: \mathcal{M}^C \to \mathcal{M}$$
is a left Quillen functor. It is required to show that the colimit functor preserves weak equivalences. 
Every weak equivalence in $\mathcal{M}^C$ can be written as a composition of a projective trivial cofibration 
and a pointwise trivial fibration. The colimit of a projective trivial cofibration is a trivial cofibration 
in $\mathcal{M}$ because $\mathrm{colim}_C$ is left Quillen. Since both domains and codomains of the morphisms 
in $\mathcal{I}$ are $\kappa$-presentable, it follows that the class of trivial fibrations $\mathcal{I}^{\square}$ 
is closed under $\kappa$-filtered colimits. This means that $\mathrm{colim}_C$ preserves trivial fibrations too, 
and then the result follows.  
\end{proof}

We now discuss our main result on an accessibility estimate for the class of weak equivalences. The proof of this requires additional 
assumptions on the combinatorial model category. 

First, we recall that given an object $X$ in a model category, a cylinder object for $X$ is a factorization of the codiagonal morphism
$$X \bigsqcup X \xrightarrow{\quad (i_0,i_1) \quad} \cyl(X) \xrightarrow{\quad p \quad} X$$
such that $(i_0,i_1)$ is a cofibration and $p$ is a weak equivalence. Cylinder objects produce factorizations of morphisms
via the classical mapping cylinder construction. In detail, given a morphism $f: X \to Y$ between cofibrant objects, the mapping cylinder 
$M_f$ is defined as the pushout
\[
\xymatrix{
& X \ar[r]^f \ar[d]^{i_0} & Y \ar[d] \ar[dr]^{\mathrm{id}} &\\
X \ar[r]^(.4){i_1} & \cyl(X) \ar@/_1pc/[rr]_{fp} \ar[r]^{\overline{f}} & M_f \ar[r]^{q} & Y
}
\]
and $X \stackrel{\overline{f}i_1}{\longrightarrow} M_f \stackrel{q}{\longrightarrow} Y$ is the mapping cylinder factorization of $f$ into a cofibration followed by a weak equivalence.

Secondly, we introduce some terminology. We say that a morphism $i: A \to B$ in a model category is a \textit{formal cofibration} if every pushout diagram
$$
\xymatrix{
A \ar[r]^i \ar[d]^f & B \ar[d]^g \\
X \ar[r] & Y
}
$$
where $X$ is cofibrant, is also a homotopy pushout. In particular, if $f$ is a weak equivalence, then so is $g$. It is well-known that every cofibration
with cofibrant domain is a formal cofibration (see, e.g., \cite[Proposition A.2.4.4]{L}). Moreover, if the model category is left proper, then every cofibration 
is a formal cofibration. 

\begin{theorem} \label{thmA}
Let $\mathcal{M}$ be a $\kappa$-combinatorial model category. Suppose that the following are satisfied:
\begin{itemize}
\item[(i)] there is a set $\mathcal{A}$ of $\kappa$-presentable cofibrant objects such that every object in $\mathcal{M}$ 
is a $\kappa$-filtered colimit of object in $\mathcal{A}$. 
\item[(ii)] every $\kappa$-presentable cofibrant object in $\mathcal{A}$ has a (not necessarily functorial) cylinder object which is again $\kappa$-presentable.
\item[(iii)] if a composite morphism $i' = q i$ is a cofibration between cofibrant objects, then $i$ is a formal cofibration. 
\end{itemize}
Then the full subcategory $\cw$ of $\mathcal{M}^{\to}$ spanned by the class of weak equivalences is $\kappa$-accessible.
\end{theorem}
\begin{proof}
By Proposition \ref{acc-embedded}, the class of weak equivalences in $\mathcal{M}$ admits $\kappa$-filtered colimits. Thus, it suffices to show that every weak equivalence is a $\kappa$-filtered colimit of weak equivalences between $\kappa$-presentable objects. 

Let $f$ be a weak equivalence. Let $\cm^{\to}_{\kappa}$ denote the full subcategory of morphisms between $\kappa$-presentable objects in $\cm^{\to}$, $\cm^{\to, \mathrm{cof}}_{\kappa}$ the full subcategory of morphisms between $\kappa$-presentable cofibrant objects, and $w\cm^{\to}_{\kappa}$ the full subcategory of weak equivalences between $\kappa$-presentable objects. There are inclusions of slice categories as follows,
$$
\xymatrix{
(\cm^{\to, \mathrm{cof}}_{\kappa} \downarrow f) \ar[rd]^{U} & & \\
& (\cm^{\to}_{\kappa} \downarrow f) \ar[r]^(.6){F} & \cm^{\to} \\
(w \cm^{\to}_{\kappa}  \downarrow f) \ar[ur]_{V} & &
}
$$
where $F$ is the canonical diagram of $f$ with respect to $\cm^{\to}_{\kappa}$. This diagram is $\kappa$-filtered and 
has colimit $f$. By (i), it is easy to see that the functor $U$ is cofinal. Then the composite diagram $F \circ U$ is $\kappa$-filtered and also has $f$ as a colimit. Analogously, it suffices to show that $V$ is cofinal.

More precisely, it suffices to show that every commutative square 
\begin{equation}  \label{square-problem}
\xymatrix{
K \ar[r]^u \ar[d]^g & X \ar[d]^f \\
L \ar[r]^v & Y
} \tag{\textasteriskcentered}
\end{equation}
where $K$ and $L$ are $\kappa$-presentable and $f$ is a weak equivalence, admits a factorization as 
follows
\[
\xymatrix{
K \ar[r] \ar[d]^g & K^\prime \ar[d]^h \ar[r] & X \ar[d]^f \\
L \ar[r] & L^\prime \ar[r] & Y
} 
\]
where $h$ is a weak equivalence between $\kappa$-presentable objects. The difference with Corollary \ref{triv-cof-fact} is that $g$ is not necessarily a cofibration. 

In fact, we can assume that the objects $K$ and $L$ are in $\mathcal{A}$. To see this, note that having a square \eqref{square-problem}, we can use (i) to factorize 
$$u:K \xrightarrow{\quad u_1 \quad} K_1 \xrightarrow{\quad u_2 \quad} X$$
where $K_1$ is $\kappa$-presentable and in $\mathcal{A}$. Then $L^\prime$ in the pushout
\[
\xymatrix{
K \ar[r]^{u_1} \ar[d]^g & K_1 \ar[d]^{\overline{g}} \\
L \ar[r]^{\overline{u}_1} & L^\prime
}
\]
is $\kappa$-presentable with a unique morphism $w:L^\prime\to Y$ such that $w\overline{u}_1=v$ and $w\overline{g}=fu_2$. Using (i) again, we get
a factorization
$$w:L^\prime \xrightarrow{\quad w_1 \quad} L_1 \xrightarrow{\quad w_2 \quad} Y$$
where $L_1$ is $\kappa$-presentable and in $\mathcal{A}$. For $g^\prime=w_1\overline{g}$, we get a factorization
\[
\xymatrix{
K \ar[r]^{u_1} \ar[d]^g & K_1 \ar[d]^{g^\prime} \ar[r]^{u_2} & X \ar[d]^f \\
L \ar[r]_{w_1\overline{u}_1} & L_1 \ar[r]_{w_2} & Y
} 
\]
 
Thus, we may restrict to squares \eqref{square-problem} with $K$ and $L$ in $\mathcal{A}$. Then consider the mapping cylinder factorization of $g: K \to L$, 
$$K \stackrel{i}{\to} M_g:= \cyl(K) \cup_K L \stackrel{q}{\to} L.$$
Since $K$ and $L$ are cofibrant, $i$ is a cofibration and $q$ is a weak equivalence. Moreover, by (ii), $M_g$ is $\kappa$-presentable and cofibrant. 

There are functorial (cofibration, trivial fibration)-factorizations of the maps $u$ and $vq$ and therefore a factorization of the morphism $(u,vq): i \to f$, 
\[
\xymatrix{
K \ar@{>->}[r]^{u'} \ar[d]^i & X^\prime \ar[d]^{f'} \ar@{->>}[r]^{\sim} & X \ar[d]^{f} \\
M_g \ar@{>->}[r]^{l'} & Y^\prime \ar@{->>}[r]^{\sim} & Y
} \tag{I}
\]
where $u'$ and $l'$ are cofibrations. By the ``2-out-of-3'' property, the morphism $f'$ is again a weak equivalence. By Corollary \ref{triv-cof-fact}, there is a 
factorization
\begin{equation} \label{first-fact}
\xymatrix{
K \ar[r]^{u_1} \ar[d]^i & K^\prime \ar[d]^j \ar[r]^{u_2} & X' \ar[d]^{f'} \\
M_g \ar[r]^{l_1} & L^\prime \ar[r]^{l_2} & Y'
} \tag{II}
\end{equation}
where $u' = u_2 u_1$, $l'= l_2 l_1$ and $j$ is a trivial cofibration between $\kappa$-presentable objects. Therefore we have a diagram as follows, 
\[
\xymatrix{
& K \ar[d]^i \ar[r]^{u_1} & K' \ar[d]^j \ar[r]^{u_2} & X' \ar[r]^{\sim} \ar[d]^{f'} & X  \ar[d]^{f} \\
\cyl(K) \ar[r]^{\overline{g}} \ar[d]^p & M_g \ar[d]^{q}_{\sim} \ar[r]^{l_1} & L^\prime \ar[r]^{l_2} & Y' \ar[r]^{\sim} & Y \\
K \ar[r]^{g} & L \ar@/_1pc/[rrru]_{v} &&& 
} \tag{III}
\]
where the square on the left is a pushout. To obtain the required factorization, it remains to extend 
the factorization \eqref{first-fact} along the weak equivalence $q: M_g \to L$. We consider the pushout $L''$:
\[
 \xymatrix{
\cyl(K) \ar[d]_{\sim} \ar[r] & M_g \ar[r]^{l_1} \ar[d]_{\sim}^{q} & L' \ar[d]^{j'} \ar[rd] \\
 K \ar[r] & L \ar[r]^{v_1} & L'' \ar[r] & Y
}\tag{IV}
\]
Since $l': M_g \to L' \to Y'$ is a cofibration between cofibrant objects, by construction, it follows by (iii) that the morphism 
$l_1: M_g \to L'$ is a formal cofibration. Therefore  the morphism $j'$ is again a weak equivalence. 

Therefore we obtain the required factorization
\[
\xymatrix{
K \ar[d]^g \ar[r]^{u_1} & K' \ar[d]^{j'j} \ar[r] & X \ar[d]^{f} \\
L \ar[r]^{v_1} & L'' \ar[r] & Y 
}\tag{V}
\]
\end{proof}

\begin{remark}(acyclic objects)
Under certain assumptions, it was shown in \cite[5.6]{MRV} that every acyclic (=weakly trivial) object in a $\kappa$-combinatorial 
model category $\cm$ is a $\kappa$-directed colimit of $\kappa$-presentable acyclic objects, i.e., that the category of acyclic objects is $\kappa$-accessible. 
The assumptions were that the terminal object $1$ is $\kappa$-presentable and the canonical morphism $X \to 1$ splits by a cofibration 
for every acyclic object $X$ in $\cm$. In $\mathcal{SS}et$, this recovers a theorem of Joyal and Wraith \cite{JW} -- any acyclic simplicial set is a 
directed colimit of finite acyclic simplicial sets. It seems that this result about objects is easier to prove than the result about 
weak equivalences, but at the same time, the former result does not follow from the latter. This can be demonstrated by the following elementary example:
let $\mathcal{S}et$ be the category of sets and consider the category $\mathcal{S}et^{\omega^{\mathrm{op}}}$ with the discrete model structure where
every morphism is a cofibration and the weak equivalences are the isomorphisms. This model category is finitely combinatorial and the terminal 
object $1$ is the only acyclic object. Since $1$ is not finitely presentable in $\mathcal{S}et^{\omega^{\mathrm{op}}}$, it cannot be a directed colimit 
of finitely presentable acyclic objects. On the other hand, the identity morphism of $1$ is clearly a directed colimit of weak equivalences between finitely 
presentable objects.
\end{remark}

\section{Examples}

We mention some immediate corollaries of Theorem \ref{thmA}. First, we consider the standard model category of simplicial sets $\mathcal{SS}et$
(see, e.g., \cite{Ho}). The cofibrations are the monomorphisms, the weak equivalences are the maps whose geometric realizations are weak homotopy
equivalences and the fibrations are the Kan fibrations. 

\begin{corollary} \label{corA}
The full subcategory $\mathcal{W}$ of $\mathcal{SS}et^{\to}$ spanned by the class of weak equivalences is finitely accessible.
\end{corollary}
\begin{proof}
It is well-known that the model category is finitely combinatorial and the standard cylinder functor $X \mapsto X \times \Delta^1$ preserves 
finitely presentable objects. The cofibrations are the monomorphisms and therefore every object is cofibrant. Then the assumptions of Theorem 
\ref{thmA} are clearly satisfied and the claim follows. 
\end{proof}

The same argument applies more generally to Cisinski model categories. A \textit{Cisinski model category} is a combinatorial model category 
whose underlying category is a Grothendieck topos and whose cofibrations are the monomorphisms. A systematic study of these model structures 
can be found in \cite{C}. We note that every Grothendieck topos $\ck$ is locally presentable and its class of monomorphisms is cofibrantly 
generated by a set of morphisms (see \cite[1.29]{C}, \cite[1.12]{B}). The following is an immediate application of Theorem \ref{thmA} to 
general Cisinski model categories. 

\begin{corollary} \label{corB}
Let $\mathcal{M}$ be a Cisinski model category. Suppose that $\mathcal{M}$ is $\kappa$-combinatorial and there is a cylinder functor 
$\cyl: \mathcal{M} \to \mathcal{M}$ which preserves $\kappa$-presentable objects. Then the full subcategory $\mathcal{W}_{\mathcal{M}}$ of $\mathcal{M}^{\to}$ 
spanned by the class of weak equivalences is $\kappa$-accessible.
\end{corollary}

For a ring $R$, let $\mathrm{Ch}(R)$ denote the projective model category of chain complexes of (left) $R$-modules \cite[Theorem 2.3.11]{Ho}. The 
weak equivalences are the quasi-isomorphisms and the fibrations are the surjections. A map is a cofibration if it is 
degreewise a split monomorphism and has
cofibrant cokernel. Every bounded chain complex of finitely generated projective $R$-modules is finitely presentable and cofibrant. We refer to \cite{Ho} for the proofs of these assertions. 

We recall the classical cylinder construction in $\mathrm{Ch}(R)$. Let $C(\Delta^1)_{\bullet}$ denote the chain complex which is concentrated in degrees 
$0$ and $1$ and defined by 
$$C(\Delta^1)_0 : = R \oplus R$$
$$C(\Delta^1)_1 : = R$$
$$\partial: C(\Delta^1)_1 \to C(\Delta^1)_0, \ \ x \mapsto (x, -x).$$
Let $S^0(R)_{\bullet}$ denote the chain complex concentrated in degree $0$ and defined by $R$. We have an obvious factorization of the codiagonal,
$$S^0(R)_{\bullet} \oplus S^0(R)_{\bullet} \to C(\Delta^1)_{\bullet} \to S^0(R)_{\bullet}.$$
More generally, given a chain complex $C_{\bullet}$, there is a factorization of the codiagonal as follows,
\begin{equation} \label{cx-cyl} 
C_{\bullet} \oplus C_{\bullet} \stackrel{i}{\to} \cyl(C_{\bullet}) : =  C(\Delta^1)_{\bullet} \otimes C_{\bullet} \stackrel{q}{\to} C_{\bullet}
\tag{Cyl}
\end{equation}
where $\otimes$ is the tensor product of chain complexes. More explicitly, $\cyl(C_{\bullet})_n$ is the $R$-module 
$C_n \oplus C_{n-1} \oplus C_n$ and the boundary map is defined by 
$$\partial(c_n, c_{n-1}, c'_n) = (\partial(c_n) + c_{n-1}, - \partial(c_{n-1}), \partial(c_n) - c_{n-1}).$$ 
The difficulty with the application of Theorem \ref{thmA} in this case is that cofibrant objects do not generate the whole category in general. 
As a result, we can conclude the finite accessibility of quasi-isomorphisms only for special cases of rings. 

\begin{corollary} \label{corC}
Let $R$ be a semi-simple ring. Then the full subcategory $\mathcal{W}_{\mathrm{qis}}$ of $\mathrm{Ch}(R)^{\to}$ spanned by the class of quasi-isomorphisms 
is finitely accessible.
\end{corollary}
\begin{proof}
It is well-known that $\mathrm{Ch}(R)$ is finitely combinatorial (see the description in \cite[Theorem 2.3.11]{Ho}). The full 
subcategory of bounded chain complexes of finitely generated $R$-modules generates the category of chain complexes. Since $R$ is semi-simple, every $R$-module is projective, so $\mathrm{Ch}(R)$ is generated under filtered colimits by
finitely presentable cofibrant objects. Thus, condition (i) of Theorem \ref{thmA} is satisfied. For condition (ii), consider the 
cylinder construction defined above. If $C_{\bullet}$ is finitely presentable (and cofibrant), the cokernel $C[-1]$ of the 
map $i$ in \eqref{cx-cyl} is cofibrant and therefore \eqref{cx-cyl} is a cylinder object for $C_{\bullet}$. Moreover, $\cyl(C_{\bullet})$ is finitely presentable if $C_{\bullet}$ is. 
Finally, condition (iii) is satisfied since every monomorphism is a formal cofibration. 
\end{proof}

We emphasize that the finite accessibility in the general case of an arbitrary ring remains open. In connection with this, we make some comments about the existence of strong generators.  
Recall that a set of objects $\{S_{\alpha}\}$ in a ca\-te\-go\-ry $\cc$ is a strong generator if the functors $\cc(S_{\alpha}, -) : \cc \to \mathrm{Set}$ 
are jointly faithful and jointly conservative (= isomorphism-reflecting). As observed in \cite[6.2]{ARV}, 
this definition is equivalent to the one in \cite[0.6]{AR} (see also \cite[6.3]{ARV}). If a category $\ck$ is $\kappa$-accessible, then in 
particular it has a strong generator consisting of $\kappa$-presentable objects. Conversely, every cocomplete category with a strong generator of $\kappa$-presentable 
objects is $\kappa$-accessible (see \cite[1.20]{AR}) and, as a consequence, it is $\lambda$-accessible for all $\lambda\geq\kappa$.

It is easy to see that $\mathcal{W}_{\mathrm{qis}} \subset \mathrm{Ch}(R)^{\to}$ has a strong generator consisting of finitely presentable objects. An example 
is the union of the following sets of morphisms: 
\begin{itemize}
 \item Let $S^{n}(R)_{\bullet}$ be the chain complex which is concentrated in degree $n$, $n \in \mathbb{Z}$, and $S^n(R)_n = R$. Consider the identity maps of 
 these chain complexes. 
 \item Let $D^n(R)_{\bullet}$ be the chain complex which is concentrated in degrees $n$ and $n-1$ and the only non-trivial boundary map is the identity map of 
 $R$. Consider the maps $0 \to D^n(R)_{\bullet}$ for all $n \in \mathbb{Z}$. 
\end{itemize}

However, having a strong generator of $\kappa$-presentable objects is weaker than $\kappa$-accessibility in general. One would still need to demonstrate 
that for some strong generator the canonical diagrams with respect to this set of objects are $\kappa$-filtered (cf. \cite[1.20]{AR}). 

\begin{remark} Since it is not true in general that a $\kappa$-accessible category is 
also $\lambda$-accessible (see \cite[2.11]{AR}), one cannot expect that any category with $\kappa$-filtered colimits and a strong generator consisting of $\kappa$-presentable objects is $\kappa$-accessible. 
This abstract argument does not work for $\kappa=\aleph_0$. For an example in this case, consider the full subcategory $\ck$ of the category of sets with monomorphisms as morphisms, 
and consisting of the sets which are either one-element or infinite. This category has filtered colimits and a one-element set is finitely presentable and forms 
a strong generator. However, $\ck$ is not finitely accessible.
\end{remark}

The next example concerns the stable module category of $R$-modules for a Frobenius ring $R$. Let $\modR$ denote the category of $R$-modules. If $R$ is 
Frobenius, this has a model structure where the cofibrations are the monomorphisms, the fibrations are the epimorphisms and the weak equivalences are 
the stable equivalences \cite[Theorem 2.2.12]{Ho}. This is a combinatorial model category, but it is not finitely combinatorial in general. We refer to \cite{Ho} for a detailed account. 

\begin{corollary} \label{corD}
Let $R$ be a Noetherian Frobenius ring. Then the full subcategory $\mathcal{W}_{\mathrm{st}}$ of $\modR^{\to}$ spanned by the class of stable equivalences
is finitely accessible. 
\end{corollary}
\begin{proof}
The model category $\modR$ is finitely combinatorial by \cite[Theorem 2.2.12]{Ho}. Since every object is cofibrant, condition (i) of Theorem \ref{thmA} is satisfied. 
(iii) is obvious, so it remains to prove (ii). Following \cite[Lemma 4.2]{KR}, cylinder objects for an object $K$ are defined by objects of 
the form $K^\ast\oplus K^\ast\oplus K$ where $K^\ast$ is a weak reflection of $K$ to an injective object in $\modR$. Since injective objects coincide with projective 
objects and, following \cite[Theorem 3.1.17]{EJ}, they coincide with flat objects, given such a weak reflection $K\to K^\ast$ of a finitely presentable $K$, there is 
a factorization through a finitely presentable injective object $(K^{\ast})_f$. Such an object can be taken as a weak reflection of $K$, and then the associated 
(non-functorial) cylinder object $\cyl(K) = (K^\ast)_f \oplus (K^\ast)_f \oplus K$ is finitely presentable, as required.
\end{proof}

Finally we comment on the cases of diagram model categories and (left) Bousfield localizations. Given a $\kappa$-combinatorial model $\mathcal{M}$ category and $C$ a small category, we can consider the diagram category $\mathcal{M}^C$ with the projective model structure where the weak equivalences and the fibrations are defined 
pointwise (see, e.g., \cite[11.6]{H}). We recall that this model category structure is lifted 
from the product model category structure on $\mathcal{M}^{\mathrm{Ob}(C)}$ along the adjunction of the restriction functor $u^*$,
$$u_!: \cm^{\mathrm{Ob}(C)} \rightleftarrows \cm^{C}: u^*.$$
The resulting model category is again $\kappa$-combinatorial, but there are difficulties with applying Theorem \ref{thmA} mainly because of condition (i). We do not know under what reasonable assumptions on 
$\cm$ and $\cw$, the category of pointwise weak equivalences $\cw_{\mathcal{M},C} \subset (\cm^{C})^{\to}$ is $\kappa$-accessible for all $C$. Note that the latter category is the same as $\cw^C$.

Concerning the weaker property of having a strong generator of $\kappa$-presentable
objects, the problem becomes much easier. Suppose that the category of weak equivalences $\cw \subset \cm^{\to}$ has a strong generator $\mathcal{W}_0$ consisting of weak equivalences between $\kappa$-presentable 
cofibrant objects. Then the subcategory of (pointwise) weak equivalences $\cw_{\mathcal{M}, \mathrm{Ob}C}$ in $(\cm^{\mathrm{Ob}C})^{\to}$ has a strong generator of $\kappa$-presentable 
cofibrant objects given by the collection of tuples $(f_i)_{i \in \mathrm{Ob}C}$ such that $f_i$ is the initial object except for $< \kappa$ entries where it is in $\cw_0$. Since 
$u^*$ is faithful and detects isomorphisms, it 
is easy to check that the image of this strong generator under $u_!$ consists of weak equivalences between $\kappa$-presentable objects and forms a strong generator for the category 
of pointwise weak equivalences $\cw_{\mathcal{M},C} \subset (\cm^{C})^{\to}$.

Regarding left Bousfield localizations (see \cite{H} for a comprehensive account), we recall that given a left proper $\kappa$-combinatorial model category $\cm$ and $S$ a set of morphisms, then the left Bousfield localization $L_S \cm$ of $\cm$ at $S$ exists and is again a combinatorial model category (see \cite[A.3.7]{L} for example). Since this has the same cofibrations, it follows that the class of $S$-local equivalences is again closed under $\kappa$-filtered 
colimits by Proposition \ref{acc-embedded}. Theorem \ref{thmA} has the following immediate corollary. We emphasize, however, that a major difficulty in estimating 
the accessibility rank of $S$-local equivalences is, of course, to specify generating sets of trivial cofibrations in $L_S \cm$. 

\begin{corollary} \label{corE}
Let $\mathcal{M}$ be a left proper $\kappa$-combinatorial model category satisfying the assumptions (i) and (iii) of 
Theorem \ref{thmA} and $S$ a set of morphisms in $\cm$. 
Suppose that the left Bousfield localization $L_S \cm$ of $\cm$ at $S$ is $\lambda$-combinatorial, $\lambda \geq \kappa$, 
and satisfies (ii) of Theorem \ref{thmA}. Then the full subcategory $\cw_S$ of $L_S \cm^{\to}$ spanned by the $S$-local equivalences is $\lambda$-accessible. 
\end{corollary}
\begin{proof}
It suffices to show that conditions (i)-(iii) of Theorem  \ref{thmA} are satisfied. (i) and (ii) hold by assumption. (iii) follows from the analogous 
property in $\mathcal{M}$. Indeed every formal cofibration $i: A \to B$ in $\mathcal{M}$ is also a formal cofibration in $L_S \mathcal{M}$ because
$\mathcal{M} \to L_S \mathcal{M}$ preserves homotopy pushouts.  
\end{proof}

\end{document}